\documentclass[aap,preprint]{imsart}
\RequirePackage{amssymb,amsthm}
\RequirePackage{graphicx}
\RequirePackage{enumerate}

\arxiv{arXiv:1201.1573}

\startlocaldefs
\newtheorem {theorem}{Theorem} 
\newtheorem {result}{Result} \newtheorem{proposition}[theorem]{Proposition} 
\newtheorem {definition}{Definition} 
\newtheorem* {corollary1}{Corollary 1.1} 
\newtheorem {corollary}{Corollary}[theorem] 
\newtheorem {lemma}{Lemma}[theorem] \newtheorem*{lemma*}{Lemma} \newtheorem*{tool*}{Lemma} 
\newtheorem {conjecture}{Conjecture}  \newtheorem*{remark*}{Remark}
\newtheorem {remark}{Remark} \newtheorem {example}{Example} 
\newtheorem {hyp}{Hypothesis}
\newcommand{\bthrm}{\begin{theorem}} \newcommand{\ethrm}{\end{theorem}}
\newcommand{\bdf  }{\begin{definition}} \newcommand{\edf}{\end{definition}}
\newcommand{\bcor }{\begin{corollary}} \newcommand{\ecor}{\end{corollary}}
\newcommand{\bres }{\begin{result}} \newcommand{\eres}{\end{result}}
\newcommand{\blem }{\begin{lemma}} \newcommand{\elem}{\end{lemma}}
\newcommand{\bconj}{\begin{conjecture}} \newcommand{\econj}{\end{conjecture}}
\newcommand{\brem }{\begin{remark}} \newcommand{\erem}{\end{remark}}
\newcommand{\bex  }{\begin{example}} \newcommand{\eex}{\end{example}}
\newcommand{\beq  }{\begin{equation}} \newcommand{\eeq}{\end{equation}}
\newcommand{\bea  }{\begin{eqnarray}} \newcommand{\eea}{\end{eqnarray}}
\newcommand{\beaw  }{\begin{eqnarray}} \newcommand{\eeaw}{\end{eqnarray}}
\newcommand{\ben  }{\begin{enumerate}} \newcommand{\een}{\end{enumerate}}
\newcommand{\bpf  }{\begin{proof}} \newcommand{\epf}{\end{proof}}

\newcommand{\eqref}[1]{(\ref{#1})}

\def\ln {{\mbox{ln}}}  

 \def\cal{\mathcal}
  \def\E {{\mathbb E}} \def\F {{\mathcal F}} 
  \def\P {{\mathbb P}}  
\def\R {{\mathbb R}} \def\Rp {{\mathbb R_{+}}}

  \def\CC {{\mathcal C}}  
 \def\G {{\mathcal G}}    \def\MM
{{\mathcal S}}   \def\MM{{\mathcal M}} 
\def\X {{X }}

\def\ss {{\subset}} \def\-{^{-1}}  \def\=> {{\ \Rightarrow \ }} \def\<=> {{\
\Leftrightarrow \ }}

\def\tri {{\triangle}} 	
\def\vsi {{\varsigma}}

\def\blambda{\bar{\lambda}}  \def\bnu{\bar{\nu}}
\def\tlambda{\widetilde{\lambda}} \def\tilth{\widetilde{h}} \def\tiltg{\widetilde{g}} 
 
\def\tmu{\widetilde{\mu}}  \def\tS{\widetilde{S}} \def\tg{\widetilde{g}}

\endlocaldefs

\begin{document} 

\begin{frontmatter}
\title{On Stability of Hawkes Process
}
\runtitle{On Stability of Hawkes Process}

\begin{aug}
\author{\fnms{Dmytro} \snm{Karabash}\thanksref{t11}
	\ead[label=a11]{karabash@cims.nyu.edu}\ead[label=a12]{dmytro.karabash@gmail.com}}
\thankstext{t11}{Partially supported by NSF grants:DMS-0904701, DMS-1208334 and DARPA.}
\affiliation{Courant Institute of Mathematical Sciences, New York University}

\address{Dmytro Karabash\\
         Courant Institute of Mathematical Sciences\\ 
	 New York University\\
	 251 Mercer St\\
	 New York, NY, 10012\\
	 \printead{a11}\\\phantom{E-mail:\ }\printead*{a12}}

\end{aug}

\begin{abstract}
  Existence and stability properties are studied for Hawkes process, i.e.  point process $S$ that has long-memory and intensity   
  $r(t)=\lambda \big(g_0(t)+ \sum_{\tau<t, \tau \in S} h(t-\tau) \big)$.  The  approach to Hawkes process presented in
  this paper allows us to prove the uniqueness of invariant distribution of the process under weaker conditions.  New speed
  of convergence results are also shown.  Unlike previous results the function $\lambda$ is  not required to be
  Lipschitz and   can be  even discontinuous. Some  generalizations are also considered.
\end{abstract}



\end{frontmatter}

\section{Introduction}
\subsection{Definition}\hspace{0pt} 
Hawkes Process is a
time-homogeneous self-exciting locally-finite point process 
on $\R$ with long-memory.  A realization of the process is a random locally-finite subset $S=S(\omega)$ of $\R$.
Any locally-finite point process is characterized by its intensity rate $r(t,\omega)$ defined as the intensity
of the point process at time $t$ conditioned on the past history of the process until time $t$, i.e. $r(t,\omega)$ can
be defined as 
\begin{equation}	
  r(t,\omega):=\lim_{\delta t \to 0} \frac{\P\bigg[\# \big( S(\omega)\cap [t,t+\delta t) \big) \geq 1 | {\cal
  F}_t\bigg]}{\delta t}	
\end{equation}
where $\# \big( S(\omega)\cap [t,t+\delta t) \big)$ is the number of elements of the set $S(\omega)\cap
[t,t+\delta t)$ and ${\cal F}_t$ is the $\sigma-$field for $S_t(\omega):=S(\omega) \cap (-\infty,t)$ generated by
{\it elementary} events
$\{\omega: S_t(\omega)\cap I \neq \emptyset \}$
where $I$ varies over all intervals $I\subset (-\infty, t)$.

It is convenient to view the process as a random subset of $\Rp:=[0,\infty)$ 
and thus the  intensity  function $r(t,\omega)$ is defined for $t \in \Rp$. 
Hawkes Process is defined by three $\Rp$-valued functions on $\Rp$.  
Given these three functions $\lambda,h,g_0$ the
intensity rate is given by
  \begin{equation} \label{def:sigma} r(t,\omega)=\lambda \big(g_0(t)+\sum_{\tau \in S_t} h(t-\tau) \big) \end{equation}
where $g_0$ is some initial condition.

The above description is equivalent to the one  found in literature except for the term $g_0$.  
Usually instead of starting from initial condition $g_0$, the process is 
defined by giving $S_0$ a locally finite collection of points in $(-\infty,0)$. This set $S_0$ represents points of
process before time $0$. To connect to our setting $g_0$ can then be computed as 
  \beq g_0(t)=\sum_{\tau \in S_0} h(t-\tau) \label{eq:old_form} \eeq
and the the two definitions are equivalent.  Defining process with $g_0$ as initial conditions also allow for
more general functions which are not of form \eqref{eq:old_form}.

This approach, besides providing this slight generalization of the model,
leads to a somewhat different perspective: Hawkes
process is a solution of a stochastic partial differential equation with jumps.  
In fact Hawkes process corresponds to solution of one of the simplest such equations where
the state of the process at time $t$ is the random \textit {impulse function} $g_t(\cdot)$ that evolves by translation
in time $g_{t+\delta t}(\cdot)=g_t(\cdot +\delta t)$ with  jumps $g_t(\cdot)\to g_t(\cdot)+h(\cdot)$ occurring at rate
$\lambda(g_t(0))$.
There is a built-in invariance under time translation due to the form of \eqref{def:sigma}.
The evolution of the \textit{impulse function}  $g_t:\Rp \to \Rp$ is defined by:
  \begin{equation} g_t(s)=g_0(t+s)+\sum_{\tau\in S_t} h(t+s-\tau) \end{equation} 
It is  $\F_t$-measurable and is a Markov process (it need not contain all the past information as is the case  when
$h(t)=e^{-t}$). 
Given two times $s \ge t \ge 0$, the conditional distribution of the next $\tau \in S$ after time $t$
being greater than $s$ is given by
\begin{eqnarray} 
  \P[\tau \ge s|{\cal F}_t]&:=&\P\big[ S \cap [s,t)=\emptyset, S \cap [t,\infty) \neq \emptyset | \F_t \big]\\
  &=&\exp \left[-\int_0^{s-t} \lambda(g_t(s))ds \right] 
\end{eqnarray}
at which point $g_\tau(\cdot)$ jumps to $g_\tau(\cdot)+h(\cdot)$. Thus $g_t$ is by  itself a Markov process.

\subsection{Results}
The  two  main results of this paper, Theorems \ref{thrm:modulus} and \ref{thrm:jumps},  generalize the results of 
Br\'{e}maud-Massouli\'{e} \cite{Bremaud} in different directions and prove that under certain conditions on
$\lambda$ and $h$, for a certain wide class $\mathcal C$ of initial conditions $g_0$, the distributions 
of $g_t(\cdot)$ converge to a common limit.  Furthermore the limiting distribution is supported on $\mathcal C$. The
secondary new result, Theorem \ref{thrm:speed0}, is on speed of convergence. 

In addition to these theorems we also observe several facts well known for attractive systems. These are
collected in Proposition \ref{prop:convergence}. The results are stated in sections \ref{sec:existence} and \ref{sec:uniqueness}.

\subsection {History and references to related work}
Point processes were first studied in \cite{Erlang} by Erlang in connection with queueing theory.  
Hawkes process were first introduced in \cite{Hawkes} to study
self-exciting point processes.   See \cite{Cox,Liniger} for additional references.  
The current work is the first one that covers cases where  $\lambda$  need not be Lipschitz and in fact can even have
jumps.  Large deviations  questions for  wide classes of   $\lambda$ 
and  $h$  have been studied by Zhu  in  \cite{Zhu1} and \cite{Zhu2};  for the special case of  linear $\lambda$ 
explicit  large deviation rates  and other limit theorems have been derived in  \cite{KZ}.

Currently Hawkes processes are used to model many phenomena ranging from queues and population growth
to mutations and  spread of infections; to defaults and jumps in financial markets; to neuroactivity and
social-networks; and finally even to modeling of artificial intelligence and creative thinking.  
While most of the real world applications involve use of multi-dimensional Hawkes process, in order to keep the 
presentation simple,  we limit ourselves to
the one-dimensional version. Possible generalizations  are outlined in section \ref{sec:gen}.

\subsection{Example} \label{ssec:example} 
Simple example comes from population growth. 
Consider a population that grows either by immigration or by birth generated by the current population. Immigrants are
treated as newborns when they arrive.   The immigration rate is $A$ and the birth rate for each individual is
$h(s)$, which depends only on  the current age $s$ of the individual. Then total growth rate of the population at time
$t$ is then given by 
  \beq r(t)=A+\sum_{\tau \in S, \tau<t} h(t-\tau)\eeq  or if we normalize with $h$ by 
$\int_0^\infty h(s) ds=1$, we replace this by
  \beq r(t)=A+B \sum_{\tau \in S, \tau<t} h(t-\tau)\eeq
since $B$ comes out of normalization.  So we see that this is a Hawkes process with $\lambda(z)=A+Bz$.  The Hawkes
processes with such linear $\lambda$ appear very often and are related to Galton-Watson trees which in turn provides an
easy way of studying many properties of Hawkes processes.

\vskip 0pt
\section{Formal Definitions} \label{sec:def}
\subsection{Formal Definition of a Hawkes Process}  \label{ssec:def}
\begin{definition} \label{P}
  \it{Hawkes process} is a random collection $S$ of
  points in $\Rp:=[0,\infty)$ characterized by triplet of functions $(\lambda,h,g_0)$ where the functions $\lambda,
  h,g_0: [0,\infty) \to [0,\infty)$. The conditional Poisson intensity at time $t$ is
  \beq r_t:=\lambda \left( g_0(t)+\sum_{0\le \tau<t, \tau \in S} h(t-\tau) \right) \eeq where the sum is over all
  previous points $\tau$ in $S_t:=S \cap [0,t)$. The function $g_0$ describes  the initial condition and the functions
  $\lambda$ and $h$ describe the evolution of the process.  We denote this Hawkes process by quadruple
  $(\P,\lambda,h,g)$.
\end{definition}
We will define a coupling of these different measures $\P_g$ by a canonical construction in the subsection
\ref{ssec:canonical}.

It is convenient to take as  state space  the space $\X$ of locally integrable $\Rp$-valued functions on $\Rp$.
Let $\MM(X)$ be the space of probability distributions on $\X$.
Then the process can be realized as an $\X$-valued process.
Start at time $t$ from with initial condition $g_t \in \X$ and consider the random evolution determined by two
components: deterministic
flow  $g_{t+\delta t}(s)=g_0(t+\delta t+s)$ up to stopping time $\tau$ at which point $g_\tau$ jumps  
$g_{\tau+0}(x)=g_{\tau-0}(x)+h(x)$
with the
distribution of $\tau=\tau(g)$ given by
\begin{equation}
  \P_g[\tau \ge t+\delta t]=\exp\left[ -\int_0^{\delta t} \lambda(g_t(s)) ds \right].
\end{equation}
After time $\tau$, the process restarts in the sense that same procedure is to be repeated with $\tau$ as a new
starting time and $g_{\tau}$ as a new initial condition.
Each time stopping corresponds to a point in $S$.   The generator of the semi-group $T_t$ of the Markov process acting
on functions $F: \X \to \Rp$  is given by:
  \beq \mathcal A:=\mathcal D+\lambda(g(0))(\theta_h-\mbox{Id}) \label{eq:generator}\eeq
where $\theta_h$ is a shift operator defined by $(\theta_h)F(g)=F(g+h)$ and $\mathcal D$ is the  derivative of
push-forward of time evolution defined by
  \beq \mathcal D F(g):=\lim_{\epsilon \to 0} \frac{F(\sigma_{\epsilon} g)-F(g)}{\epsilon} \eeq
where $\sigma_{\epsilon}$ is the time-shift operator defined by $\sigma_{\epsilon} g(s)=g(s+\epsilon)$.
Then this $\X$-valued process $g_t$ satisfies  
\begin{equation} \label{eq:g_t(s)}
  g_t(s)=g_0(t+s)+\sum_{\tau \in S_t} h(t+s-\tau).
\end{equation}

For any initial condition $g_0(\cdot)$, we have a  Markov
process $g_t(\cdot)$. 
Then $\lambda(z_t)$ will be the intensity of the point process, where
  \begin{equation} z_t:=g_t(0)=g_0(t)+\sum_{\tau \in S_t} h(\tau -t). \end{equation}  

It is sometimes more convenient to consider instead of $S_t$ the process 
  \beq Q_t=(N_t,g_t) \eeq
where $N_t=\#(S_t)$ is the number of jumps from time $0$ to time $t$.  Then $Q_t$
can be viewed as a point process $N_t$  driven by the  Markov process  $g_t$. 

\begin{remark}
  Without loss of generality we can assume $\int_0^{\infty} h(t) dt=1.$
  Indeed one can always achieve this if $\|h\|_1=\int h(t) dt<\infty$ by observing that triples
  $\big(\lambda(z),h(t),g(t)\big)$ and $\left(\lambda(z \|h\|_1),\frac{h(t)}{\|h\|_1}, \frac{g(t)}{\|h\|_1}\right)$
  produce the same Hawkes process. On the other hand if $\| h \|_1=\infty$ and $\inf_{z \in \Rp} \lambda(z)>0$ 
 then $\lim_{t \to \infty} g_t(0)=\infty$ a.s. and hence there can be  no  stationary version of the process. 
\end{remark}

\subsection{Convergence notions} We equip the space $X$ with $L_1^{loc}$ metric
\beq \forall g,f \in X, \qquad d_X(g,f)=\sum_{i=1} \frac 1{2^n} \cdot \frac{\int_0^{n} |g(s)-f(s)|ds}{1+\int_0^{n} |g(s)-f(s)|ds} \eeq
The space $\MM(X)$  will then have the topology of  weak convergence  inherited from the metric space $X$.
Let $D(X)$ be the  space of $X$-valued functions on $[0,\infty)$ equipped with Skorohod $J_1$-metric $d$ 
and let $\G$ be the $\sigma$-field  generated by the 
open sets in $(\X, d)$. Collection $\{\G_t\}$ is the corresponding filtration.
Let $g^*$ be the non-initial part of impulse function:
\begin{equation} \label{eq:g_t^*(s)}
  g_t^*(s)=\sum_{\tau \in S_t} h(t+s-\tau)
\end{equation}
and 	
\begin{eqnarray}
  \forall g, \tg \in D(X), \qquad d^*(g,\tg):=d(\mathcal S(g),\mathcal S(\tg))=d(g^*,\tg^*) 
\end{eqnarray}
be the pull-back of $d$ under the map $\mathcal S:D(X)\to D(X)$, $\mathcal S: g \mapsto g^*$ defined according to \eqref{eq:g_t^*(s)} as
  \beq \big(\mathcal S(g)\big)_t(s)=g_t(s)-g_0(t+s) \eeq

Let $\mu_{f,s}$ be the distribution of $g_s$---impulse-function at time $s$---starting from initial condition 
$f$.

The distribution of the whole process $g$ starting with initial condition $g_0$
 will be denoted by $\P_{g_0}$ and the associated expectations  will be denoted by $\E_{g_0}$.  Any $\mu \in \MM(X)$
can be viewed as a random initial condition and
  \beq \P_{\mu}:=\int \P_f \mu(df), \quad \E_{\mu}:=\int \E_f \mu(df), \eeq
denote the corresponding probability measure and expectation with respect to it.
Similarly $\P^*$ and $\E^*$ will be used as analogous expressions for the $g^*$.


In Theorems \ref{thrm:modulus} and \ref{thrm:jumps} we consider total variation $d_{TV}$
of the difference of $\P^*$ starting from two different initial conditions which can also be viewed as total variation of
the corresponding $\P$ but under a different $\sigma$-field $\G^*=\G \circ \mathcal S^{-1}$.  

Furthermore let us  define $d_{TV,T}$ to be the  total variation after time $T$, i.e. the  total variation on the  sigma-algebra
$\G \circ \big(S|_{[T,\infty)} \big)^{-1}$ where operator $(S|_{[T,\infty})$ is defined by:
  \beq \big(\mathcal S|_{[T,\infty)}(g)\big)_t(s)=1_{t>T} \big(g_t(s)-g_0(t+s)\big) \eeq

\section{Tools} \label{sec:tools}
\subsection{Canonical Construction} \label{ssec:canonical}
One way to construct point processes is to start from a  canonical  Poisson point process $\P$.  To any measure $m$ one
can associate a Poisson  point process such  that for  any measurable set $A$, the  number of points in $A$ will be
random variable with  a Poisson distribution with mean  $m(A)$. In our case canonical  Poisson point process  will
correspond to the choice $m(A)=|A|$, the Lebesgue measure on $\Rp \times \Rp$.   Given $\lambda,h,g$ measure $\P$
induces a measure $\P_{g}^{\lambda,h}$ consistent with the definition above in the following way. If $S$ is our point
process then $\tau\in S$
if there is a point on the vertical line  ${\tau} \times [0,\lambda(z_t))$ of our plane $\Rp \times \Rp$. In the rest of
the paper we will denote $\P_g^{\lambda,h}$ by $\P_g$  because  $\lambda$ and $h$ would be the same unless stated
otherwise.

This induces a natural coupling of the collection of the measures $(\P_g)$ indexed by initial conditions $g$ that we
will call {\it canonical coupling}.  This coupling is no way unique but this particular coupling is maximal  when 
$\lambda$ is  non-decreasing and we have two different initial conditions  one of which is strictly larger than the
other (by point-wise partial ordering).  The coupling is maximal in the sense that   the set of points common to both 
of these two processes will  stochastically dominate the similar set of points in  any other coupling.

\subsection{Coupling / Stochastic Domination}
Here we present the canonical  coupling that we will  use. 
\begin{tool*}[Stochastic Domination]
  Suppose $(S,\lambda,h,g_0)$ and $(\tS, \tlambda,\tilth,\tiltg_0)$ are two Hawkes processes satisfying 
  \beq \forall x \in \Rp, (h(x) \leq \tilth(x), g(x) \leq \tiltg(x)) \label{cond:order}\eeq 
  \beq \forall x \leq y, \lambda(x) \leq \tlambda(y) \label{cond:midlambda}\eeq
  Then there exists a coupling such that  $S \ss \tS$.\\
  Note: $h, \tilth$ are not necessarily normalized.
\end{tool*}

\begin{proof}
  If $a(t)$, $b(t)$ are two intensities and if $a(t)\le b(t)$ for all $t$, one can couple the point processes,
  such that the two processes jump together with rate $a(t)$ and the second one jumps by itself at rate $b(t)-a(t)$.
  This can also be done if $\forall \omega \in \Omega, a(t,\omega) \leq b(t,\omega)$.  (In fact this is done naturally
  by our construction that is discussed in the previous subsection.)  Hence it remains to prove that 
    \beq \lambda(g_t(0)) \leq \tlambda(\tiltg_t(0)) \label{eq:domination}.\eeq

  We know that up to the first jump in $\tS$ we have $g_t(0)=g_0(t), \tiltg_t(0)=\tiltg_0(t)$.  However we know from 
  \eqref{cond:order} and \eqref{cond:midlambda} that 
    \beq \lambda(g_0(t)) \leq \tlambda(\tiltg_0(t))\eeq
  which in turn implies that \eqref{eq:domination} holds up to first jump of $\tS$.
  Then we claim by induction that it holds for all times. Indeed at the time of jump the order of 
  $g \leq \tiltg$ is preserved because there is no jump in $S$
  before first jump of $\tS$ and  the jump of  $\tiltg$ is always larger because $h \leq \tilth$. 
\end{proof}

\subsection{Parent-Offspring Structure and the Branching representation}
We add the following parent-offspring structure to obtain
a random forest structure embedded in time. This will be  useful due to its connection to Galton-Watson trees (branching
processes).

We start with the Hawkes process corresponding to  functions $(\lambda,h,g_0)$ and let $S=\{\tau_1,\tau_2,...\}$, where
$\tau$ is an increasing sequence, i.e. $i<j \rightarrow \tau_i < \tau_j$. 
Let us also denote 
\beq \lambda_0(z):=\lambda(z)-\lambda(0) \eeq
Now to each $\tau_i$ we associate a randomly chosen parent element $p(\tau_i)$ from $S \cup \{-\infty\}$, where
$-\infty$ represents having no parent and being a root node.  
Given a sequence $(\tau_1,\tau_2,...,\tau_i)$,   we define $\{p(\tau_i)\} $ to be  mutually independent with  distributions given by :
\begin{eqnarray}
   \P[p(\tau_i)=-\infty|\tau_1,...\tau_i]&=&
   \frac{\lambda(0)}{\lambda(z_{\tau_i})}+\frac{\lambda_0(z_{\tau_i})}{\lambda(z_{ \tau_i})}\frac{g(\tau_i)}{z_{\tau_i}}
\\ \P[p(\tau_i)=\tau_j|\tau_1,...\tau_i] &=&
    1_{j<i} \frac{\lambda_0(z_{\tau_i})}{\lambda(z_{ \tau_i})}\frac{h(\tau_i-\tau_j)}{z_{\tau_i}}
\end{eqnarray}
Note that $\{p(\tau_i)\}$ while   mutually independent are  not identically distributed.

Figure \ref{HawkesColor} provides a particular visualization of parent-offspring  structure that is consistent with the above description.
\begin{figure} \label{HawkesColor}
  \begin{center}
    \includegraphics[scale=0.5]{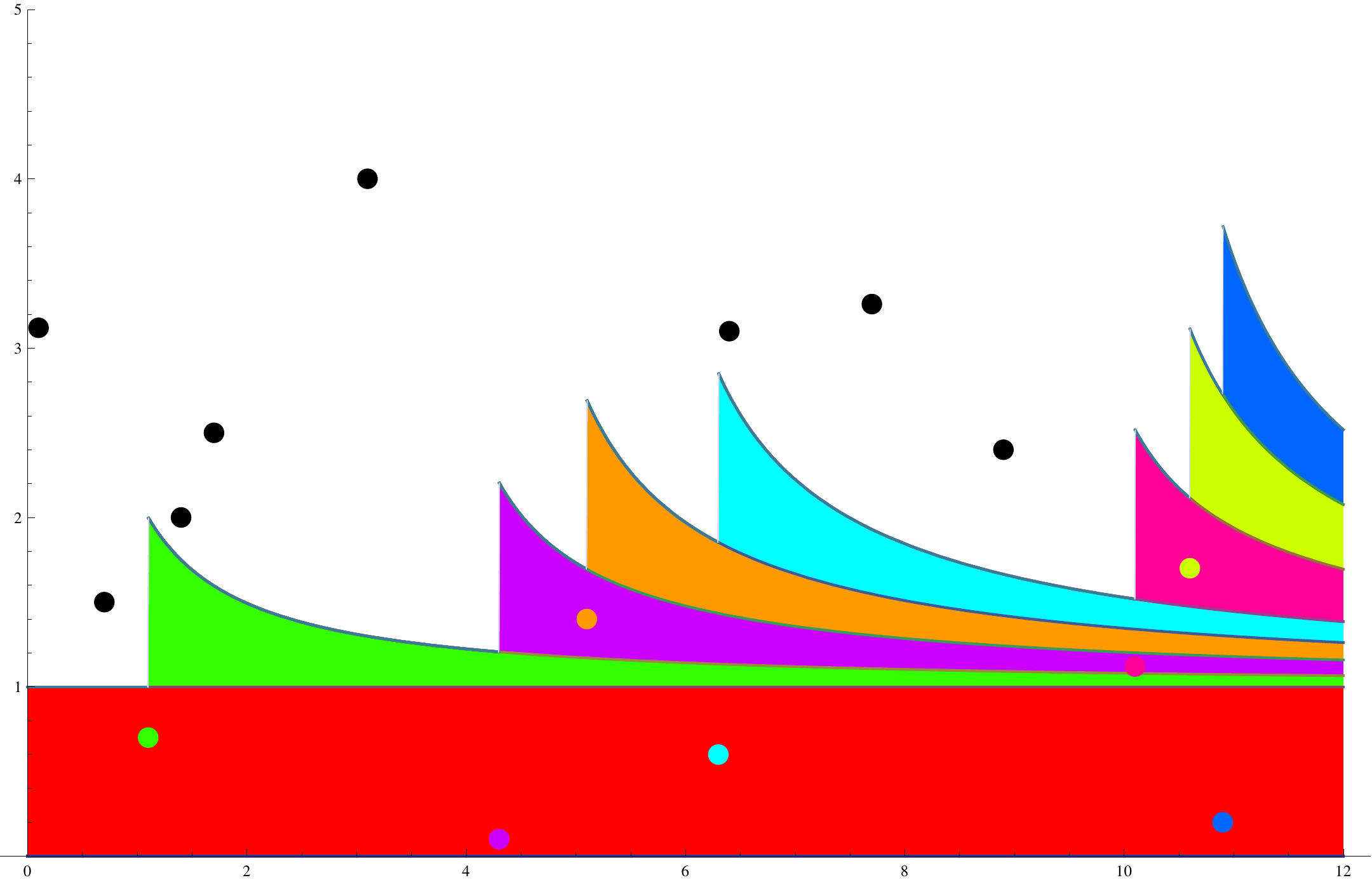}
    \caption{For both $\lambda(z)=1+z, \quad h(t)=\frac{1}{(1+t)^2}$.  Each region corresponds to children of point of
    corresponding color}
  \end{center}
\end{figure}
We are specifically interested in linear case which provides us with dual description of the same process as described 
in the following tool:

\begin{tool*}[Branching Process Equivalence]
  When $\lambda(z)=\blambda(z)=A+Bz$ roots have rate $A+Bg(t)$ and each tree is a branching process with
  the number of branches having a Poisson distribution with mean $B$. The distribution of age $\tau_j$ of any parent  at
  the time $\tau$ of birth of any child  is then given by
    \beq \label{shape}  \P[ \tau-\tau_j >t | p(\tau)=\tau_j]=\int_t^{\infty} h(s) ds.\eeq
\end{tool*}
\begin{proof}
  By the above scheme we see that the intensity of roots is given by
  \beq \P[p(\tau_i)=-\infty]\lambda(z_{\tau_i})=\lambda(0)+Bz_{\tau_i}\frac{g(\tau_i)}{z_{\tau_i}}=A+Bg(\tau_i). \eeq
  Now each new point creates an area of size $B \|h\|_1=B$ and hence the number of children is Poisson$(B)$.  Now the
  shape of
  the area is given via $h$ which then proves \eqref{shape}.
\end{proof}
\vskip 0pt
\section{Existence} \label{sec:existence}
We begin with preliminary results regarding existence and ergodicity. 
Let us consider the following hypothesis on Hawkes Process.

\begin{hyp}\label{hyp:existence}
  The function  $\lambda$ satisfies: 
  \begin{equation} \label{cond:wd}
    \exists A,B \geq 0, \forall z \in \R, \lambda(z) \leq \blambda(z):=A+Bz 
  \end{equation}
\end{hyp}
which brings us to our first proposition:
\begin{proposition} \label{prop:convergence}
  Suppose Hawkes process satisfies Hypothesis~\ref{hyp:existence}. Then the following
  three statements hold:
  \begin{enumerate}[(i)]
  \item Hawkes process is well defined for all times.
  \item  When the Hypothesis \ref{hyp:existence} is satisfied with $B<1$, there exists an invariant distribution
    for process $g_t$.
  \item If in addition $\lambda$ is non-decreasing and we start from 
    $0$ initial condition, i.e. $g_0=0$, then the distribution of $g_t$ 
    converges weakly to $\mu_0$, which being the unique minimal invariant measure is ergodic.
  \end{enumerate}
\end{proposition}

\begin{proof}[Proof of Proposition \ref{prop:convergence}]  (i) Consider Hawkes process with $\blambda(z)=A+Bz$. By
  Galton-Watson representation it is well defined for all
  times. It can be coupled  with our  process  which it will dominate. Hence the dominated process is also  well defined
  for all times.
  
  (ii) If $B<1$, then the dominating process has  uniformly bounded  density; then so does the dominated process and we
  can get an invariant distribution $\mu_{g_0}$ by taking Cesaro limit along a subsequence 
    \beq \label{eq:mu} \mu_{g_0}:=\lim_{t_k \to \infty} \frac 1{t_k} \int_0^{t_k} \mu_{s,g_0} ds \eeq
  where $\mu_{s,g_0}$ represents marginal distribution of the impulse-function at time $s$ given the initial condition
  $g_0$.

  (iii) The expected size of each tree is finite and hence we
  have bounded  density.  Since $\lambda$ is non-decreasing the process is order-preserving. Hence 
    \begin{equation} \E_0\big[ f(g(t+s_1),\ldots, g(t+s_k)) \big] \end{equation}
  is non-decreasing in $t$ for any non-decreasing $f$ which implies convergence of $\mu_{s,0}$ to $\mu_0$. 

  This gives us weak convergence.  We claim  that  invariant distribution $\mu_0$
  is  the unique minimal one, i.e. stochastically dominates any other invariant distribution.  To see this consider any
  other invariant $\tmu$ and choose initial condition chosen randomly  according  to  $\tmu$; but then we have
  point-wise domination initially and by coupling we see that  there is domination at all times and letting
  $t\to\infty$,  we see that $\tmu$ dominates $\mu$.  Hence $\mu$, being the unique   minimal  invariant measure, is
  extremal. Hence it  is ergodic.
\end{proof}

\section{Uniqueness} \label{sec:uniqueness}
\begin{definition} \label{stable}
  Let $\mu$ be an  invariant distribution supported on  a class of functions $\CC$, meaning that
    \beq \mu(\CC)=1 \eeq
  Then we say that the pair $(\lambda,h)$ is {\it
  $(\CC,\mu)$-stable} if starting from any initial condition $g_0$ in $\CC$ the distribution $\mu_{t,g_0}$ of the
  impulse
  function at time $t$ converges to $\mu$:
    \beq \lim_{t \to \infty} \mu_{t,g_0} \mathop{=}^d \mu \eeq
  where $\mathop{=}^d$ signifies that the limit is in the  sense of weak convergence. 
\end{definition}

Now we turn to main results of this paper.  Here is our $2^{nd}$ hypothesis:

\begin{hyp}\label{hyp:H}
  The function $\lambda$ is non-decreasing and it satisfies 
  \begin{equation} \label{cond:modulus}
      \sup_{x \in \Rp} \left( \lambda(x+s)-\lambda(x) \right)\le  \phi(s) 
    \end{equation}
    for some concave non-decreasing $\phi$ satisfying 
    \begin{equation} \label{cond:phiH}
      \int_{0}^{\infty} \phi(H(s))ds=C<\infty, \quad \mbox{where } H(s)=\int_s^\infty h(t)dt
    \end{equation} 
\end{hyp}

\begin{theorem} \label{thrm:modulus}
  If Hawkes process satisfies both Hypothesis \ref{hyp:existence} with $B<1$ and Hypothesis \ref{hyp:H} then
  pair $(\lambda,h)$ is \it{$(\CC,\mu)$-stable} as in definition \ref{stable} 
  with 
  \begin{equation}
      \CC:=\left\{ g \in C(\Rp): \int_0^{\infty} \phi(g(s))ds< \infty \right\}
  \end{equation} 
  and $\mu$ being $\mu_0$ from Proposition \ref{prop:convergence} part (iii).   
\end{theorem}

\begin{remark}
  In this theorem we relax the Lipschitz condition with constant 1 on $\lambda$ that was imposed in \cite{Bremaud} (in a
  slightly different form since 
  $\| h \|$ was not normalized).  When we do have this assumption the results of \cite{Bremaud} follow as a corollary of
  the above theorem \ref{thrm:modulus}.  In fact the results follow even under weaker assumption that $\lambda$ is
  Lipschitz with constant $L$ for some $L$: 
  \beq \forall x,y \in \Rp |\lambda(x)-\lambda(y)|  \leq L |x-y| \eeq  
\end{remark}

  Before we start with proof of theorem \ref{thrm:modulus} let us prove the following lemma:

\begin{lemma} \label{lemma:mean-field}
  If $\mu$ is a stationary distribution of the impulse function $g_t$ of a Hawkes process then 
  \beq \E_{\mu}[g(s)]=\E_{\mu}[\lambda \big(g(0) \big)]H(s). \eeq 
\end{lemma}

\begin{proof} By stationarity
  \begin{eqnarray}
    \nonumber \E_{\mu}[g(s)]  &=&\E_{\mu}\left[ \int_s^{\infty} \lambda\big(g(0))h(t) dt \right]
			    \\&=&\E_{\mu}[\lambda \big(g(0) \big)]\int_s^{\infty} h(t) dt
			    \\&=&\E_{\mu}[\lambda \big(g(0) \big)]H(s).\nonumber 
  \end{eqnarray}
\end{proof}

\begin{proof}[Proof of Theorem \ref{thrm:modulus}:]
  We use a recurrence argument. It is enough to show that exists some class $\CC' \ss \CC$ such that:

  (i) $f \in \CC'$ implies that $\P^*_f$ and $\P^*_0$ have non-trivial overlap
      \beq I(\P^*_f, \P^*_0):=1-d_{TV}(\P^*_f, \P^*_0)\geq \delta >0 \eeq

  (ii) $\CC'$ is {\it recurrent with respect to $\CC$}: starting from any point in class $\CC$ 
    the impulse function will enter $\CC'$ at some time in the future, i.e. 
      \beq \forall g_0 \in \CC, \P_{g_0}[\forall t \in \Rp, g_t \notin \CC' ]=0 \eeq

  We will split the proof into three steps.  In step (i) and (ii) we will show the above statements
  and in step (iii) we will complete  the proof using recurrence.  Step (i) in turn will suggest the choice of $\CC'$.

  Step (i): First observe that by applying Jensen's inequality and concavity of $\phi$ we obtain:
    \begin{eqnarray}
      I(\P^*_f, \P^*_0)&=&\E_{0}\left[ \exp \left(-\int_0^{\infty} |\lambda(z_t+f(t))-\lambda(z_t)| dt \right) \right]
      \\&\geq& \exp \left(-\E_{0}\left[\int_0^{\infty} |\lambda(z_t+f(t))-\lambda(z_t)| dt \right]\right)
      \\&\geq& \exp \left(-\E_{0}\left[\int_0^{\infty} \phi(f(t)) dt \right] \right)
      \\&=& \exp \left(-\int_0^{\infty} \phi(f(t)) dt \right)
    \end{eqnarray}

  Step (ii): Step (i) suggests that  we take
    \beq \label{CC'} \CC':=\left\{ g \in C(\Rp): \int_0^{\infty} \phi(g(s))ds< K \right\} \eeq 
    whereas for $\CC'$ to satisfy the recurrence condition we will make  a suitable choice of  $K$.  We know that for
    any finite time $t$ process which started from $\CC$ will remain in $\CC$ since $\phi$ is convex.  Now
    consider starting from $\CC$ but replacing  $\lambda$  by with $\blambda$ in the definition  of our process.  
    Then, because $\phi$ is nondecreasing   by applying stochastic domination   we get
    \begin{eqnarray} \label{eq:theorem3start}
      \E^{\lambda}_{g} \left[\int_0^{\infty} \phi(g_t(s))ds \right] &\leq&  \E^{\blambda}_{g} \left[\int_0^{\infty}
      \phi(g_t(s))ds \right].
    \end{eqnarray}
    We now take  $\limsup_{t \to \infty}$ on  both sides of \eqref{eq:theorem3start} and  by proposition
    \ref{prop:convergence}(iii),   the
    right hand side converges to a limit from any initial $g$.

    \begin{eqnarray} 
      \overline{\lim_{t \to \infty}} \E^{\lambda}_{g} \left[\int_0^{\infty} \phi(g_t(s))ds \right]
      &\hspace{-10pt} \leq&  \lim_{t \to \infty} \E^{\blambda}_{g} \left[\int_0^{\infty} \phi(g_t(s))ds \right] \\
      &\hspace{-10pt} \leq&  \lim_{t \to \infty} \E^{\blambda}_{0} \left[\int_0^{\infty} \phi(g_t(s))ds \right] \\
      \label{eq:theorem3end2}&\hspace{-10pt} \leq&  \lim_{t \to \infty} \int_0^{\infty} \phi( 
      \E^{\blambda}_{0}[\lambda(g_t(0))]H(s))ds 
    \end{eqnarray}
    where the line \eqref{eq:theorem3end2} follows from Lemma \ref{lemma:mean-field}.  
    Recall that part of assumption of hypothesis 2 is that $\phi$ is concave and increasing. While  concavity  implies
    that $\phi(cx) \leq c \phi(x)$ for $c \geq 1$, monotonicity implies  that $\phi(cx) \leq \phi(x)$ for $c<1$. 
    Combining these two we get $\phi(cx) \leq \max(c,1) \phi(x)$. Applied to \eqref{eq:theorem3end2} with
    $c=\E^{\blambda}_{0}[\lambda(g_t(0))]$ we obtain:

    \begin{eqnarray}
      \label{eq:theorem3end} \label{eq:K}
      \overline{\lim_{t \to \infty}} \E^{\lambda}_{g} \left[\int_0^{\infty} \phi(g_t(s))ds \right] 
      &\hspace{-10pt} \leq& \hspace{-5pt} C  \max( \lim_{t \to \infty} \E^{\blambda}_{0}[\lambda(g_t(0))],1)   
    \end{eqnarray}
    where $C=\int_0^{\infty} \phi(H(s))ds<\infty$ by Hypothesis \ref{hyp:H}.
    Hence \eqref{eq:K} is finite since  $\lim_{t \to \infty} \E^{\blambda}_{0}[\lambda(g_t(0))]<\infty$
    by branching representation. Hence $\mu(\CC)=1$ and in \eqref{CC'} it is enough to set
    \begin{eqnarray}
      K=\max( \lim_{t \to \infty} \E^{\blambda}_{0}[\lambda(g_t(0))],1) \times \big[ \int_0^{\infty} \phi(H(s))ds\big].
    \end{eqnarray}

  Step (iii): Hence we return almost surely  to the class $\CC'$.  Now starting from any initial condition in $\CC'$,
    \begin{eqnarray} \label{stepiiibound}
      I(\P^*_f, \P^*_0)&\geq&1-\exp \left(-K \right)
    \end{eqnarray}
    and result follows from the following schematic representation
      \beq \label{eq:collapsed} g \rightarrow \CC  \rightleftarrows \CC' \rightarrow \mu \eeq
    where all arrows represent transitions with uniformly positive probability 
    which implies the convergence to $\mu$.  

    Let us describe this in detail. Define two sequences of alternating  stopping times as follows:
    Let stopping time $\vsi_1$ be the first time that $g_t$ is in $\CC'$.
    Then we try to couple the process $S$ with $S^{(1)}$ where $S^{(1)}$ is Hawkes process with same pair $(\lambda,h)$
    but that starts at time $t$ with $g_t=0$; we use canonical coupling.  
    Then by step (i) the coupling is successful with probability at least $\delta>0$.
    If it is not then there is a stopping time $\upsilon_1 \in S \setminus S^{(1)}$, in this case we
    restart the procedure defining $\vsi_i, \upsilon_i$ by the following recursive definition
    \begin{eqnarray}
      \vsi_1		&=&\inf \{t: g_t \in \CC'\} \nonumber \\
      \forall i \geq 1, \upsilon_i 	&=&\inf \{t>\vsi_i: t \in S \setminus S^{(k)}\} \label{eq:stimes}\\
      \forall i >1,	\vsi_i		&=&\inf \{t>\upsilon_{i-1}: g_t \in \CC' \} \nonumber
    \end{eqnarray}
    where $S^{(k)}$ is Hawkes Process that starts at time $\vsi_k$ with condition $g_t=0$. 
    Hence we want to show that almost surely for some $i$, $\vsi_i=\infty$ which is immediate from step (i):

    \begin{eqnarray}
      \P[\forall i, \vsi_i<\infty] 
      &=&\E[\prod_{i=1}^{\infty} I(\P^*_{g_{\tau_i}}, \P^*_0)]\\
      &\leq& \prod_{i=1}^{\infty}  \big( 1-e^{-K} \big)=0.
    \end{eqnarray}
\end{proof}

Finally we turn to the last theorem where $\lambda$ can have jumps. 

\begin{hyp} \label{hyp:jumps} 
  (i) Function $\lambda$ is non-decreasing and satisfies:
    \begin{equation} \label{cond:jumps:l} 
      \mbox{ $\lambda(0)>0$, $\lambda \leq \blambda$, $B<1$}
    \end{equation}
  (ii) Function $h$ is convex and satisfies: 
    \begin{equation} \label{cond:jumps:h}
      \mbox{$\|h\|_{\infty}<\infty$, $\log |h'(x)|=o(x)$}, 
      \int_0^{\infty} t h(t) dt<\infty 
    \end{equation}
\end{hyp}

\begin{remark}
  These assumptions allow large tails in $h$ which was the aim of this study, in     
  particular $h(x)=\frac{p}{(1+x)^{p+1}}$ where $p>0$. 
\end{remark}
\begin{remark}
  These assumptions can be weakened to allow piece-wise decreasing $h$ as long as $h^{-1}(y)$ is finite for any $y \in
  \Rp$.
\end{remark}

\begin{theorem} \label{thrm:jumps}
  If Hawkes process satisfies Hypothesis \ref{hyp:existence} with $B<1$ and 
  Hypothesis \ref{hyp:jumps} then pair $(\lambda,h)$ is \it{$(\CC,\mu)$-stable} as in definition \ref{stable} 
  with 
  \begin{equation} \label{def:Class2} 
    \CC:=\left\{ g \in C(\Rp):  \int_0^{\infty} t g(t)dt<\infty \right\} 
  \end{equation}
  and $\mu$ being $\mu_0$ from Proposition \ref{prop:convergence} part (iii).   
\end{theorem}

\begin{definition} \label{def:nu}
  Let $\nu$ denote distribution of $g(0)$ under $\mu$.  
\end{definition}

Next we state estimates on the tail of $\nu$ as well as on the probability density of $\nu$.
We will prove the following two  Lemmas after we complete  the proof of Theorem \ref{thrm:jumps}.

\begin{lemma}  \label{lemma:total} Under Hypothesis
    \ref{thrm:jumps} the following statements hold: 
    \begin{equation} \label{eq:exp_mu}
      \exists \theta>0, \E_{\mu}[e^{\theta z_t}]<\infty,
    \end{equation}
    In particular $\nu$ has exponential tail: there exist constants   $0< c_1<\infty$ and $\theta_1>0$  such that for
    all  $z\ge 0$,
    \begin{equation}
      \nu[z,\infty) \leq c_1 e^{-\theta_1 z}, 
    \end{equation}
\end{lemma}

\begin{lemma} \label{lemma:local} Under Hypothesis \ref{hyp:jumps}, there exists a constant $c_1'$ such that for any invariant distribution
$\mu$, the associated $\nu$ (distribution of $g(0)$ under $\mu$),
  has a density $\psi(z)$ that satisfies  
  \begin{equation}
    \psi(z) \leq c'_1 \lambda(z) \nu[z+h(0), \infty)
  \end{equation} 
  Furthermore there exist constants  $c_2,\theta_2$ such that
  \begin{equation}
    \psi(z) \leq c_2 e^{-\theta_2 z}
  \end{equation}
\end{lemma}

\begin{proof}[Proof of Theorem \ref{thrm:jumps}:]
  Unlike proof of Theorem \ref{thrm:jumps} the comparison is done between initial conditions $f$ and $g+f$ where $g$ is
  randomly chosen with distribution $\mu_0$. We denote  by $\mu^\ast_f$ the distribution of $g+f$, i.e. $\mu_0$
  shifted    by $f$, i.e. $\mu^\ast_f(A):=\mu_0 \big(\{g-f: g \in A\} \big)$.
  The idea of the  proof remains the same as that of  Theorem \ref{thrm:jumps} while the calculations
  of \eqref{eq:theorem3start}-\eqref{eq:theorem3end} are redone as follows.

  \begin{eqnarray}
    I(\P^*_f, \P^*_{\mu_f^*})&=&\E_{\mu_0}\left[ \exp \left(-\int_0^{\infty}
    \lambda \big(z_t+f(t) \big)-\lambda(z_t) dt \right) \right]
    \\&\geq& \exp \left(-\E_{\mu_0}\left[\int_0^{\infty} \lambda \big(z_t+f(t) \big)-\lambda(z_t) dt \right]\right)
  \end{eqnarray}
  by Jensen's Inequality.  Since  $\mu_0$  is invariant we can replace $z_t$ by $z_0$:
  \begin{eqnarray}
    E_1&:=&\E_{\mu_0}\left[\int_0^{\infty} \lambda \big(z_t+f(t) \big)-\lambda(z_t) dt \right]
    \\&=& \E_{\mu_0}\left[\int_0^{\infty} \lambda \big(z_0+f(t) \big)-\lambda(z_0) dt \right]
    \\&=& \mathop{\mathop{\int \int \int}_{z,x,t \in \Rp}}_{z \leq x \leq z+f(t) } \psi(z) d\lambda(x)  dz dt
    \\&=& \mathop{\mathop{\int \int \int}_{z,x,t \in \Rp}}_{x-f(t) \leq z \leq x } \psi(z) dz d\lambda(x)  dt
  \end{eqnarray}
  We can now apply the estimate on $\psi(z)$ from  Lemma \ref{lemma:local} and obtain
  \begin{eqnarray}
    E_1&\leq& \mathop{\mathop{\int \int \int}_{z,x,t \in \Rp}}_{x-f(t) \leq z \leq x }  c_2 e^{-\theta_2 z} dz
d\lambda(x)  dt
    \\&\leq& \mathop{\mathop{\int \int \int}_{z,x,t \in \Rp}}_{-f(t) \leq z \leq 0 }  c_2 e^{-\theta_2 x} e^{-\theta_2 z} dz d\lambda(x)  dt
    \\&\leq& \mathop{\mathop{\int \int \int}_{z,x,t \in \Rp}}_{-f(t) \leq z \leq 0 }  
	      c_2 e^{-\theta_2 x} e^{\theta_2 f(t)} dz d\lambda(x)  dt d\lambda(x)  dt
    \\&=&  c_2 \|f e^{\theta_2 f} \|_{L_1} \mathop{\int}_{x \in \Rp} e^{-\theta_2 x} d\lambda(x)
    \\&\leq&  c_2 \|f\|_{L_1} e^{\theta_2 \|f\|_{L_{\infty}}} \mathop{\int}_{x \in \Rp} e^{-\theta_2 x} d\lambda(x)
    \\&<&\infty
  \end{eqnarray}
  where the last line follows from Hypothesis \ref{hyp:jumps} since $f \in L_1 \cap L_{\infty}$; last term  is integrated by parts:
  \begin{eqnarray}
    \mathop{\int}_{x \in \Rp} e^{-\theta_2 x} d\lambda(x)=\lambda(0)+
    \mathop{\int}_{x \in \Rp} \theta_2e^{-\theta_2 x} \lambda(x)dx<\infty
  \end{eqnarray}

  The rest of the proof follows analogously to Theorem \ref{thrm:modulus}.
\end{proof}

\subsection{Lemmas for Theorem \ref{thrm:jumps}}

The two lemmas in this section provide estimates on the
density as well as the tail  probabilities of distribution $\nu$.
Let us first prove the tail estimate.
\begin{proof}[Proof of Lemma \ref{lemma:total}:] 
  Observe first that $\nu[z,\infty] \leq \bnu[z,\infty]$ by stochastic domination. We can now  use Galton-Watson representation in the
  following way. 
  Consider a  tree  $T$  viewed as a population starting with one individual born at time $0$. 
  Each individual born at time $s$  gives  birth to new individuals at  rate   $Bh(t-s)$. Since $B<1$ and $\int h(s) ds=1$, the size of the
  tree $T$ is finite. Consider the quantities:
  \begin{eqnarray}
    \label{def:Lambda_T} \Lambda_{\theta}(t)&:=&\ln\, \E [\exp( \theta \sum_{\tau \in T \atop \tau\le t}h(t-\tau))]\\
    \label{def:Lambda} \Lambda_{\theta}&:=&\ln \,\E_{\mu}[\exp(\theta z_t)]
  \end{eqnarray}

  If we ignore time,  we have a Galton-Watson tree with the number of branches having a Poisson distribution with
  parameter $B<1$.
  Then it is known that (see \cite{Nakayama} for example), the total size $Z$ of the population has an exponential moment, i.e.
  \beq \E [e^{\theta Z}]<\infty \eeq
  for some $\theta_0>0$. In particular since $\|h\|_\infty<\infty$, exists $\theta=\frac{\theta_0}{\|h\|_\infty}$
  \beq \E \big[\exp( \theta \sum_{\tau \in T\atop\tau\le t}h(t-\tau) \big]\le \E[\exp(\theta \|h\|_\infty Z)]<\infty \eeq
  By dominated convergence theorem 
  \beq \lim_{t\to\infty} \Lambda_{\theta}(t)=0\eeq

  Observe that for linear $\blambda(z)=A+Bz$
  \beq \exp[\sum_{\tau\in T\atop\tau\le t} \theta h(t-\tau)-\theta h(t)-B\int_0^t [e^{\Lambda_\theta(t-s)}-1]h(s)ds] \eeq
  is a martingale. Therefore $\Lambda_\theta(t)$ satisfies
    \beq \label{eq:Lambda_T} 	\Lambda_{\theta}(t)=\theta h(t)+B\int_0^t (e^{\Lambda_{\theta}(t-s)}-1)h(s)ds \eeq
  Similarly
    \beq \label{eq:Lambda} 	\Lambda_{\theta }=A \int_0^{\infty} (e^{\Lambda_{\theta}(s)}-1)ds<\infty \eeq
  as we add the contributions of  $e^{\Lambda_T(s)}-1$ from  trees at $-s$ that are created at rate $A$.  It only remains
  to show that right hand side of \eqref{eq:Lambda_T} is finite.

  Integrate \eqref{eq:Lambda_T} with respect to $t$ from $0$ to $\infty$. 
  \begin{eqnarray}
    \int_0^{\infty} \Lambda_\theta(t)  &\le& \theta +B \int_0^{\infty} dt\int_0^t  (e^{\Lambda_{\theta}(t-s)}-1)h(s)dsdt\\
    &=& \theta + B\int\int_{0\le s\le t<\infty} (e^{\Lambda_{\theta}(t-s)}-1) h(s)dtds\\
    &\le &\theta+ B\int_0^{\infty}  (e^{\Lambda_{\theta}(t)}-1)dt
  \end{eqnarray}
  Since $\Lambda_\theta(t)\to 0$ as $t\to\infty$, given $\delta>0$, we can bound $ (e^{\Lambda_{\theta}(t)}-1)$ by
  $(1+\delta)\Lambda_\theta(t)$
  for sufficiently large $t$,  providing us with an estimate:
    \beq \int_0^{\infty} \Lambda_\theta(t) dt\le \theta+C_\delta+B(1+\delta)\int_0^{\infty} \Lambda_\theta(t) dt \eeq
  We choose $\delta>0$ so that $B(1+\delta)<1$. This proves that $\int_0^\infty \Lambda_\theta(t) dt$ as well as $\int_0^\infty
  (e^{\Lambda_{\theta}(t)}-1)dt$
  are finite.
\end{proof}

Having obtained exponential tail  estimates we now turn to  the proof of the density estimate:
\begin{proof}[Proof of lemma \ref{lemma:local}:] Let $L(s,t,A)$ be  time spent in $A$ by $z_t:=g_t(0)$
  between time $s$ and $t$ starting from arbitrary initial condition, i.e.
  \begin{equation}
    L(s, t,A)=\int_s^t 1_{z_s \in A}ds.
  \end{equation}
  Let $I=[z,z+\epsilon]$, $\tau_0=0$ and $\tau_i, i>0$ be $i^{th}$ jump after $0$, that is
    \begin{equation}\tau_i=\inf\{t>\tau_{i-1}: t \in S\}. \end{equation}
  Furthermore let 
    \begin{equation} n_t=\inf\{s>t: s \in S\}.  \end{equation}
  be the first jump after time $t$. We note that we can calculate $\nu(I)$ as

    \begin{eqnarray} \nu[I]&=&\lim_{t \to \infty} \frac{\E_{\mu}\left[L(0,t,I)\right]}t\end{eqnarray}
  We consider first the numerator.
    \begin{eqnarray} L(0,t,I)&\leq&\sum_{i} L(\tau_i,\tau_{i+1},I) 1_{\tau_i<t}\end{eqnarray}
  Taking expectations and using the properties of  conditional expectations :
  \begin{eqnarray}
    \E[L(0,t,I)] 	&\leq&	\E[\sum_{i} \E[L(\tau_i,\tau_{i+1},I)|\F_{\tau_i-}] 1_{\tau_i<t}]
    \\			&=&	\E[\sum_{i} F(g_{\tau_i-}) 1_{\tau_i<t}]
  \end{eqnarray}
  where function $F:X \to \Rp$ is defined as
    \begin{eqnarray} F(g)=\E_g[L(0,\tau_1,I)] \end{eqnarray}
  We   finally have
  \begin{eqnarray} \label{eq:LandF}
    \E[L(0,t,I)] \leq \E[\int_0^t F(g_{t-}) dN_t] 
  \end{eqnarray}
  which we can now view  as  a stochastic integral  of a predictable process with respect to  $dN_t$: using \eqref{eq:LandF} and the
  stationarity of $g(t)$ we obtain
  \begin{eqnarray}
    \nu[I]&=&\lim_{t \to \infty} \frac{\E_{\mu}\left[L(0,t,I)\right]}t
    \\&\leq&\lim_{t \to \infty} \frac{\E_{\mu}\left[ \int_0^t F(g_{t-}) dN_t \right] }t
    \\&\leq&\lim_{t \to \infty} \frac{\E_{\mu}\left[ \int_0^t F(g_{t-}+h) \lambda(g_t) dt \right] }t
    \\&\leq&\lim_{t \to \infty} \frac{\E_{\mu}\left[ \int_0^t F(g_{0}+h) \lambda(g_0) dt \right] }t
    \\&=&\int_{g \in X} \E_{g+h}\left[L(0,\tau_1,I)\right] \lambda(g(0))  d\mu(g)
    \\&\leq& \int_{g \in X} \E_{g+h}\left[\frac{\epsilon }{-h'(\tau_1)}\right] 1_{g(0)+h(0)>z} \lambda(g(0))  d\mu(g) \label{eq:product}
  \end{eqnarray}
  The  last equation follows from condition \eqref{cond:jumps:h}. By convexity,  $-h'$ is decreasing and  
  $L(0,\tau_1,I) \leq \frac{\epsilon}{g'(\tau_{1})}\leq \frac{\epsilon}{h'(\tau_{1})}$ since between two jumps $g(t)$  can cross
  $I$ only once. Because $\lambda(g(0))\ge A>0$,  we have the estimate  $\P[\tau_1>t]\leq e^{-At}$.
  In view of  condition  \eqref{cond:jumps:h} we obtain
    \begin{eqnarray} \E_{g}\left[\frac{\epsilon}{-h'(\tau_1)}\right] &\geq& \epsilon C_1 \end{eqnarray}
  for some constant  constant $C_1>0$. Then we continue from \eqref{eq:product}
  \begin{eqnarray}
    \nu[I]&\leq&	\int_g \E_{g+h}\left[\frac{\epsilon }{-h'(\tau_1)}\right] 1_{g(0)+h(0)>z}\lambda(g(0))  d\mu(g)
    \\	&\leq& 	\epsilon C \int_g 1_{g(0)+h(0)>z}\lambda(g(0))  d\mu(g)
    \\	&=& 	\epsilon C \int_{x>z+h(0)} \lambda(x)  d\nu(x)
    \\	&=& 	\epsilon c'_1 \lambda(z)  \nu[z+h(0),\infty)
  \end{eqnarray}
  where the last  inequality follows the fact that $d\nu(x)$ has exponential tail.  Hence we have
  \begin{equation} \label{eq:local_final}
    \nu[I]\leq \epsilon c'_1 \lambda(z) \nu[z+h(0),\infty)  
  \end{equation}
  from which we first observe that there is density $\psi$ since then $\nu[I] \leq c'_1 \lambda(z) \epsilon$ and then taking limit as
  $\epsilon \to 0$ in \eqref{eq:local_final} we obtain the estimate on the density.
\end{proof}

\section{Speed of Convergence}

In general, estimating  the speed of convergence by the methods  of Theorem \ref{thrm:modulus} and \ref{thrm:jumps}
appears to be difficult. To overcome  this we need to strengthen  Hypothesis \ref{hyp:H} to the following: 
\begin{hyp} \label{hyp:H2}
  The function $\lambda$ is non-decreasing and it satisfies 
  \begin{equation} \label{cond:modulus2}
    \sup_{x \in \Rp} \left( \lambda(x+s)-\lambda(x) \right)\le  \phi(s) 
  \end{equation}
  for some $\phi$ that satisfies
  \beq \label{eq:phi(h)} \tilde B:=B \int_0^{\infty} \phi(h(t)) dt<1. \eeq
    $\tilde h(t)=\frac{B \phi(h(t))}{\tilde B}$ that satisfies:
  \beq \label{eq:iterative} \limsup_{t \to \infty} 
    \left( \frac{\sum_{n=1}^{\infty} \tilde B^n n \int_{x>\frac{t}{2n}} {\tilde h}(x) dx}{\int_{x>t} {\tilde h}(x) dx} \right)=:C_4<\infty 
  \eeq
    In addition ${\tilde g(t)}=\phi(g(t))$ also satisfies $\|\tilde g\|_{1}=\int_0^{\infty}{ \tilde g}(t)dt<\infty$ 
\end{hyp}
\begin{remark}
  Condition \eqref{eq:iterative} can be replaced by other conditions as can  be seen from the proof.  The  important
  feature is that all nicely decaying functions like  ${\tilde h}(x)=\frac p{(1+x)^{p+1}},p>0$ satisfy  Hypothesis  \ref{hyp:H2}, while
  the following $h$ does not: $h(x)=\sum_{i=0}^{\infty} 2^{-2i-1} 1_{2^{i}<x<2^{i+1}}$ where $1_{2^{i} \leq x<2^{i+1}}$ is
  the indicator function that is $1$ if $x \in [2^i,2^{i+1})$ and 0 otherwise.
\end{remark}

\begin{theorem} \label{thrm:speed0}
  If Hawkes process satisfies Hypothesis \ref{hyp:existence} with $B<1$, and Hypothesis \ref{hyp:H2} then:
  \begin{eqnarray} \label{eq:exact}
    d_{TV,t}(\P_g, \P_0) \leq B \int_t^{\infty} [(I- BQ_{\tilde h})^{-1}{\tilde g}] (s)ds
  \end{eqnarray}
  where $Q_{\tilde h}$ stands for the  convolution operator $Q_{\tilde h}g={\tilde h}\ast g$.

  In turn this implies that tail of $d_{TV,t}(\P_g, \P_0)$ cannot be much worse than that of $\phi(h(t))$ or $\phi(g(t))$;
  more precisely:
  \begin{eqnarray} \label{eq:tail}
    \limsup_{t \to \infty} \frac{d_{TV,t}(\P_g, \P_0)}{\int_{\frac t2}^{\infty} \phi(g(s))+\phi(h(s)) ds}< \infty
  \end{eqnarray}
\end{theorem}
\begin{proof}
  Consider the canonical coupling of  the measures $\P_g$ and $\P_0$, i.e.  a pair of processes $(g,g^{(0)})$
  which start from initial conditions $g_0=g$ and $g^{(0)}_0=0$ respectively.
  Let $\tri g=g-g^{(0)}$ be the difference of impulse functions, $\tri S=S \setminus S^{(0)}$ 
  be the difference in sets and the last point in $\tri S$ being:
    \beq L:=\max \{s \in \Rp: s \in \tri S\} \eeq

  Then observe that the concavity of $\phi$ along with  $\phi(0)=0$ implies subadditivity of $\phi$, i.e. $\phi(a+b) \leq \phi(a)+\phi(b)$.
  It now follows from  \eqref{cond:modulus2} that
  \begin{eqnarray}
    \lambda(g_t(s))&-&\lambda(g_t^0(s))
    \\&\leq& \phi \big(\tri (g_t(s)) \big)
    \\ &=&\phi(g_0(t+s)+\sum_{\tau \in \tri S\atop  \tau<t} h(t+s-\tau))\\ 
    &\leq& \phi(g_0(t+s))+\sum_{\tau \in \tri S\atop \tau<t} \phi(h(t+s-\tau))
  \end{eqnarray}

  By setting $s=0$, we see  that the point process   with rate $\widetilde r_t$ at time $t$ given by
  \begin{eqnarray}
    \widetilde r_t:=\phi(g_0(t))+\sum_{\tau \in \tri S\atop \tau<t} \phi(h(t-\tau))
  \end{eqnarray}
  dominates stochastically the process $\tri  S$.  Consider the process $\tri D$ with the  following description: one begins with roots that
  are created with rate $\phi(g(t))$ at time $t$ and then each root generates a Galton-Watson tree with birthrate  at time $t$ along  any
  branch born  at time $s$ given by   $\phi(h(t-s))$.  Let the last point of $\tri D$   be:
    \beq L_D=\max \{s \in \Rp: s \in \tri D\} \eeq

  Notice that
    \beq d_{TV,t}(\P_g, \P_0)=\P[L>t] \leq \P[L_D>t] \leq \E[\#(\tri D \cup [t,\infty))] \label{eq:rec:ineq} \eeq
  where the first inequality follows from $L \leq L_D$ because $\tri S \subset \tri D$ and the second 
  from  $L_D \in \tri D$. So we will be estimating  $\E[\#(\tri D \cup [t,\infty))]$ which turns out to
  be much simpler to work with.
  If we define
    \beq l( A):= \E[\#(\tri D \cap A)] \eeq
  then:
    \begin{eqnarray}
    \int_0^{\infty} e^{\theta t} l(dt)= \E[\sum_{\tau \in \tri D} e^{\theta \tau}]
    \end{eqnarray}
  Next we compute $\E[\sum_{\tau \in \tri D} e^{\theta \tau}]$ and from that can obtain  either an exact formula or estimate  the tail
behavior. Observe that we have the recurrence relation:
  \beq \label{eq:rec:tree}
    \E[\sum_{\tau \in T_{\phi}} e^{\theta(\tau-\rho)}]= 1+{\tilde B}\cdot \left(\int_0^{\infty} {\tilde B}{\tilde h}(t) e^{\theta t} dt 
    \right) \E[\sum_{\tau \in T_{\phi}} e^{\theta(\tau-\rho)}]
  \eeq
  where $1$ comes from the root; ${\tilde B}{\tilde h}(t)$ is the rate of birth of the (direct) children of the root at time $\rho+t$;
  $e^{\theta t}$ comes from time shift and finally due to the fact that each child generates independently the same random tree we multiply 
  $\E[\sum_{\tau \in T_{\phi}} e^{\theta(\tau-\rho)}]$. This in turn implies that
  \beq
  \E[\sum_{\tau \in T_{\phi}} e^{\theta(\tau-\rho)}]=  
  \frac{1}{1-{\tilde B} \cdot \int_0^{\infty} {\tilde h}(t)e^{\theta t} dt}
  \eeq

  Now similarly we let $\rho$ be an arbitrary root and $T_{\rho}$ 
  be its tree we  obtain from  \eqref{eq:rec:tree}:
  \begin{eqnarray} \label{eq:rec:final0}
  \E[\sum_{\tau \in \tri D} e^{\theta \tau}]&=&
  {\tilde B} \cdot \left(\int_0^{\infty} {\tilde g}(t) e^{\theta t} dt \right) \E[\sum_{\tau \in T_{\phi}} e^{\theta(\tau-\rho)}]
  \\\label{eq:rec:final} &=& \frac{\tilde B \cdot \left(\int_0^{\infty} {\tilde g}(t) e^{\theta t} dt \right) }{1-\tilde B \cdot
\int_0^{\infty} {\tilde h}(t) e^{\theta
  t} dt}
  \end{eqnarray}
  where the recurrence relation  in \eqref{eq:rec:final0} is similar to \eqref{eq:rec:tree}: there is no $1$ here since there is no initial
birth at time  $0$; 
  the birth rate rate is ${\tilde B}{\tilde g}(t))$ and the rest of the calculation  is the same.

  Finally when combining \eqref{eq:rec:ineq}, \eqref{eq:rec:tree} and \eqref{eq:rec:final}
  we obtain the exact formula \eqref{eq:exact}.  Observe that if
  $f_1,f_2,...f_k$  are probability density functions and $a_i \geq 0, i \in \{1,2,...k\}, \sum_{i=1}^{k} a_i=1$, then
  \begin{eqnarray}
  \int_{s=t}^{\infty} (f_1*f_2*...*f_k)(s)(ds) \nonumber
    &=&\hspace{-10pt}\mathop{\int \cdots \int}_{\sum_{i=1}^k s_i>t} f_1(s_1)...f_k(s_2) ds_1...ds_k 
  \\&\leq&\sum_{i=1}^n \mathop{\int \cdots \int}_{s_i>a_i t} f_i(s_i)ds_1...ds_k \label{eq:conv_split}
  \\&\leq&\sum_{i=1}^n \mathop{\int}_{s_i>a_i t} f_i(s_i)ds_i \nonumber
  \end{eqnarray}
  where second line follows from the fact that ${\sum_{i=1}^k s_i>t}$ implies that for some $i$, $s_i>a_it$.
  (Another way to see this is via probability interpretation if random variables $X_i$ are independent and have distribution $f_i$
  then $\P[\sum_{i=1}^{k}X_i>t] \leq \P\big[X_i>a_it \mbox{ for some }i \in \{1,2,...,k\} \big]=\sum_{i=1}^{k} \P[X_i>a_it]$ where the 
  first and last expressions correspond to those of \eqref{eq:conv_split}). Then
  \begin{eqnarray}
  \hspace{15pt} d_{TV,t}(\P_g, \P_0)
  \hspace{-5pt}&=&\int_t^{\infty} {\tilde g}* \left(\frac {1}{1-{\tilde B} \hat {\tilde h}}\right)^{\vee}(s) ds   
  \\&=&\|\tg\|_1\sum_{n=0}^{\infty} {\tilde B}^n \int_t^{\infty} \big(\frac{\tg}{\|\tg\|_1}*\underbrace{{\tilde h}*\cdots*{\tilde
h}}_{\mbox{n times}}\big)(s) ds 
  \\&\leq&\|\tg\|_1 \sum_{n=0}^{\infty} {\tilde B}^n \left( \int_{s> \frac t2} \frac{\tg(s)}{\|\tg\|_1} ds+n \int_{s>\frac t{2n}} {\tilde
h}(s) ds \right)
  \\&\leq& \frac{1}{1-{\tilde B}} \int_{s> \frac t2} \tg(s) ds+\|\tg\|_1\sum_{n=0}^{\infty} {\tilde B}^n n \int_{s>\frac t{2n}} {\tilde
h}(s) ds
  \end{eqnarray}
  using \eqref{eq:conv_split} with 
  $k=n+1, \ f_1=\frac{\tg}{\|\tg\|_1}, \ f_2=f_3=\cdots=f_{k+1}={\tilde h}, \ a_1=\frac 12,\\  a_2=a_3=\cdots=a_{k+1}=\frac 1{2n}$.
  Finally:
  \begin{eqnarray}
  &&\limsup_{t \to \infty} \frac{d_{TV,t}(\P_g, \P_0)}{\int_{\frac t2}^{\infty} \tg(s)+{\tilde h}(s)ds}
  \\&\leq&\frac{\frac{1}{1-{\tilde B}} \int_{s> \frac t2} \tg(s) ds+\|\tg\|_1\sum_{n=0}^{\infty} {\tilde B}^n n \int_{s>\frac t{2n}} {\tilde
h}(s) ds}
  {\int_{\frac t2}^{\infty} \tg(s)+\phi(h(s))ds}
  \\&\leq&\frac 1{1-{\tilde B}}+\|\tg\|_1C_4<\infty
  \end{eqnarray}
\end{proof}

\section{Generalizations} \label{sec:gen}
\subsection{Multi-type models} \label{ssec:multi-type}
This section suggests two generalizations: first making Hawkes multi-type and second replacing $h$ by a general measure.
The proofs will be  simple modification of the ones we have presented here. 

Consider that points are now of finitely many types and the set of type is $E=\{1,...,d\}$.  
Now the process is specified by analogues triple $(h,\lambda,g)$:
\begin{itemize}
 \item $h_{ij}(t): E \times E \times \R \to \Rp$ where $i,j \in E$, $\| h_{ij} \|_{L_1}=\frac 1d$
 \item $\lambda_i: E \times \mathbb R_+^{d^2} \to \Rp$
 \item $g_{ij}(t): E \times E \times \R \to \Rp$.
\end{itemize}
Then for all $e \in E$ the rate at time $t$ of type $e$ is given by 
\beq r_e(t)=\lambda_e(z(t)) \eeq 
where $z_{ij}(t)$ is a matrix given by 
\beq z_{ij}=g_{ij}(t)+\sum_{i} \sum_{\tau \in S_i(t)} h_{ij}(t-\tau) \eeq
whre $S_i(t)$ are all the point of type $i$ before time $t$.
When $h$ is a measure one extends the definition of the can extend functionals $\int_a^b r_e(s) ds$ if exist
such $k^{e}_{ij}$ for each $i,j,e \in E$ that satisfy: 
\begin{equation} \label{eq:limit}
\forall e \in, \exists k^{e}_{ij} \in \mathbb R_+^{d^2}, \lim_{c \to \infty} \frac{\lambda_e(cz)}{c}=\sum_{i,j}
z_{i,j}k^{e}_{ij}. 
\end{equation}
As the analogue of Hypothesis \ref{hyp:existence} we now have:
 \begin{equation} \label{cond:existance1}
\exists k \in \mathbb R_+^{d^3}, \forall e \in E, \lambda_e (z) \leq \blambda_e(z):=c_e+\sum_{i,j} z_{i,j}k^{e}_{ij} 
\end{equation} 

As a generalization of Propositoin \ref{prop:convergence} we get: 

\begin{corollary1} Suppose condition \eqref{cond:existance1} is satisfied and if $h$ is a measure
then condition \eqref{eq:limit} is satisfied.  Then the following statements hold:\\
(i) If $\forall e\in E$, $\lambda_e$ is non-decreasing then Hawkes process is well defined\\
(ii) If in addition matrix \beq m_{i,e}=\sum_{j \in E} k^{e}_{ij} \eeq indexed by $i,e \in E$
has spectrum supported on unit disk $\{ c \in \mathbb C: |c| < 1 \}$ then invariant measure exists 
and is unique minimal invariant measure and hence ergodic.
\end{corollary1}
\subsection{Stationary initial conditions as a generalization}  The proofs of this paper can also accommodate additional
levels of complexity
where one adds a stationary signal in the following sense.  Suppose you add another bounded stationary process $p(t)$ to the rate function
before passing through $\lambda$ so that now rate $r_t$ is defined as:
\begin{equation}
 r_t:=\lambda \left( g_0(t)+p(t)+\sum_{\tau \in S_t} h(t-\tau) \right) 
\end{equation}
The results of this paper will still hold. One can also add another signal bounded signal $q(t)$ outside of lambda where rate will be:
\begin{equation}
 r_t:=\lambda \left( g_0(t)+p(t)+\sum_{\tau \in S_t} h(t-\tau) \right)+q(t) 
\end{equation}
One useful application of this is a Hawkes-like linear process where the roots of the trees have some non-Poisson but
stationary
behavior while the trees are just like in Hawkes process.
\section*{Acknowledgements}  
I would like to thank my adviser S. R. Srinivasa Varadhan for giving this topic to me and Lingjiong
Zhu.  He not only gave us this rich and interesting topic, but also provided guidance in
motivation, formulation of results and clear view of the subject.  I also thank Lingjiong Zhu for
helpful discussions and review of my paper.

This research is supported by NSF DMS-0904701, DMS-1208334 and DARPA grants.

\end{document}